 \documentclass{monsky2009}

\usepackage{graphicx}

\usepackage{amssymb}
\usepackage{amsmath}

\usepackage{bm}

\newenvironment{theorem*}[1]{\textbf{#1}\itshape \hspace{.3em}}{\upshape}
\newenvironment{remark*}[1]{\textbf{#1}\itshape \hspace{.3em}}{\upshape}
\newenvironment{corollary*}[1]{\textbf{#1}\itshape \hspace{.3em}}{\upshape}
\newenvironment{proof}{\textbf{Proof\hspace{.3em}}}{}

\newtheorem{definition}{Definition}[section]
\newtheorem{theorem}[definition]{Theorem}
\newtheorem{lemma}[definition]{Lemma}

\newtheorem{example}[definition]{Example}
\newtheorem{corollary}[definition]{Corollary}
\newtheorem{remark}[definition]{Remark}

\newcommand{\ord}{\ensuremath{\mathrm{ord}}}

\begin{document}

\begin{frontmatter}






\title{Generating functions attached to\\ some infinite matrices}
\author{Paul Monsky}

\address{Brandeis University, Waltham MA  02454-9110, USA. monsky@brandeis.edu}

\begin{abstract}
Let $V$ be an infinite matrix with rows and columns indexed by the positive integers, and entries in a field $F$. Suppose that $v_{i,j}$ only depends on $i-j$ and is 0 for $|i-j|$ large. Then $V^{n}$ is defined for all $n$, and one has a ``generating function'' $G=\sum a_{1,1}(V^{n})z^{n}$. Ira Gessel has shown that $G$ is algebraic over $F(z)$. We extend his result, allowing $v_{i,j}$ for fixed $i-j$ to be eventually periodic in $i$ rather than constant. This result and some variants of it that we prove will have applications to Hilbert-Kunz theory.

\end{abstract}


\end{frontmatter}


\section{Introduction}
\label{sectionintro}

Throughout, $\Lambda$ is a ring with identity element 1. Suppose that $w_{i,j}$, $i$ and $j$ ranging over the positive integers, are in $\Lambda$ and that $w_{i,j}=0$ whenever $i-j$ lies outside a fixed finite set. Then if $W$ is the infinite matrix $|w_{i,j}|$, one may speak of $W^{n}$ for all $n\ge 0$, and one gets a generating function $G(W)=\sum_{0}^{\infty}a_{n}z^{n}$ in $\Lambda[[z]]$, where $a_{n}$ is the (1,1) entry in the matrix $W^{n}$. We shall prove:

\begin{theorem*}{Theorem I}
\label{theorem1.1}
Suppose that $w_{i,j}=0$ if $i-j\not\in\{-1,0,1\}$, and that $w_{i+1,j+1}=w_{i,j}$ unless $i=j=1$. Suppose further that $\Lambda =M_{s}(F)$, $F$ a field, so that $G(W)$ may be viewed as an $s$ by $s$ matrix with entries in $F[[z]]$. Then these matrix entries are algebraic over $F(z)$.
\end{theorem*}

\begin{corollary*}{Corollary}
\label{corollary1.2}
Let $F$ be a field and $v_{i,j}$, $i$ and $j$ ranging over the positive integers, be in $F$. Suppose:

\begin{enumerate}
\item[(a)] $v_{i,j}=0$ whenever $i-j$ lies outside a fixed finite set.
\item[(b)] For fixed $r$ in $Z$, $v_{i,i+r}$ is an eventually periodic function of $i$.
\end{enumerate}

Then if $V$ is the matrix $|v_{i,j}|$, the generating function $G(V)$ is algebraic over $F(z)$.

\end{corollary*}

\begin{proof}
To derive the corollary we choose $s$ so that:

\begin{enumerate}
\item[(1)] $v_{i,j}=0$ whenever $i\le s$ and $j>2s$ or $j\le s$ and $i>2s$.
\item[(2)] $v_{i+s,j+s}=v_{i,j}$ whenever $i+j\ge s+2$.
\end{enumerate}

We then write the initial $2s$ by $2s$ block in $V$ as 
$\left|\begin{smallmatrix}
D&C\\
A&B
\end{smallmatrix}
\right|$ with $A$, $B$, $C$, $D$ in $M_{s}(F)$. Our choice of $s$ tells us that $V$ is built out of $s$ by $s$ blocks, where the blocks along the diagonal are a single $D$, followed by $B$'s, those just below a diagonal block are $A$'s, those just above a diagonal block are $C$'s, and all other entries are 0. Now let $\Lambda =M_{s}(F)$ and $W=|w_{i,j}|$ where $w_{i+1,i}=A$, $w_{i,i+1}=C$, $w_{1,1}=D$, $w_{i,i}=B$ for $i>1$, and all other $w_{i,j}$ are 0. View $G(W)$ as an $s$ by $s$ matrix with entries in $F[[z]]$. One sees easily that $G(V)$ is the (1,1) entry in this matrix, and Theorem I applied to $W$ gives the corollary.\qed
\end{proof}

\begin{remark*}{Remark}
\label{remark1.3}
When $v_{i,j}$ only depends on $i-j$, the above corollary is due to Gessel. (When the matrix entries of $V$ are all\, $0$'s and $1$'s the result is contained in Corollary 5.4 of \cite{1}.  The restriction on the matrix entries isn't essential in Gessel's proof, as one can use a generating function for walks with weights.)  
\end{remark*}

Our proof of Theorem I is easier than Gessel's proof of his special case of the corollary. The reason for this is that by working over $\Lambda$ rather than over $F$ we are able to restrict our study to walks with step-sizes in $\{-1,0,1\}$. (A complication, fortunately minor, is that the weights must be taken in the non-commutative ring $\Lambda$.) Our proof is well-adapted to finding an explicit polynomial relation between $G(V)$ and $z$; we'll work out a few examples. This paper would not have been possible without Ira Gessel's input. I thank him for showing me tools of the combinatorial trade.

\section{Walks and generating functions}
\label{section2}

\begin{definition}
\label{def2.1}
If $l\ge 0$, an ordered $l+1$-tuple $\alpha = (\alpha_{0},\ldots, \alpha_{l})$ of integers is a (Motzkin) walk of length $l=l(\alpha)$ if each of $\alpha_{1}-\alpha_{0},\ldots,\alpha_{l}-\alpha_{l-1}$ is in $\{-1,0,1\}$.
\end{definition}

We say that the start of the walk is $\alpha_{0}$, the finish is $\alpha_{l}$, and that $\alpha$ is a walk from $\alpha_{0}$ to $\alpha_{l}$.

\begin{definition}
\label{def2.2}
If $\alpha$ and $\beta$ are walks of lengths $l$ and $m$, the concatenation $\alpha\beta$ of $\alpha$ and $\beta$ is the walk $(\alpha_{0},\ldots,\alpha_{l},\alpha_{l}+(\beta_{1}-\beta_{0}),\ldots,\alpha_{l}+(\beta_{m}-\beta_{0}))$ of length $l+m$.
\end{definition}

Now let $\Lambda$ be a ring with identity element 1, and $A$, $B$, $C$, $D$ lie in $\Lambda$. To each walk $\alpha$ we attach weights $w(\alpha)$ and $w^{*}(\alpha)$ in $\Lambda$:

\begin{definition}
\label{def2.3}
If $l(\alpha)=0$, $w(\alpha)=w^{*}(\alpha)=1$. If $l(\alpha)>0$, $w(\alpha)=U_{1}\cdot\ldots\cdot U_{l}$ where $U_{i}=A$, $B$ or $C$ according as $\alpha_{i}-\alpha_{i-1}$ is $-1$, $0$, or $1$. The definition of $w^{*}(\alpha)$ is the same with one change: if $\alpha_{i}=\alpha_{i-1}=0$ then $U_{i}=D$ rather than $B$.
\end{definition}

Evidently $w(\alpha\beta)=w(\alpha)w(\beta)$. Furthermore $w^{*}(\alpha\beta)=w^{*}(\alpha)w^{*}(\beta)$ whenever $\alpha$ and $\beta$ are walks from $0$ to $0$.  

\begin{definition}
\label{def2.4}
$\alpha$ is ``standard'' if each $\alpha_{i}\ge \alpha_{l}$. Note that a walk from $0$ to $0$ is standard if and only if each $\alpha_{i}\ge 0$.
\end{definition}

\begin{definition}
\label{def2.5}
$\alpha$ is ``primitive'' if $l(\alpha)>0$, $\alpha_{0}=\alpha_{l}$ and no $\alpha_{i}$ with $0<i<l$ is $\alpha_{0}$. Note that a standard walk from $0$ to $0$ is primitive if and only if $l(\alpha)>0$ and each $\alpha_{i}$, $0<i<l$, is $>0$.
\end{definition}

\begin{definition}
\label{def2.6}
\hspace{2em} 
\vspace{-2ex}
\begin{enumerate}
\item[(1)]$G(w)=\sum w(\alpha)z^{l(\alpha)}$, the sum extending over all standard walks from $0$ to $0$. $H(w)$ is the sum extending over all primitive standard walks from $0$ to $0$.
\item[(2)]$G(w^{*})$ and $H(w^{*})$ are defined similarly, using $w^{*}(\alpha)$ in place of $w(\alpha)$.
\end{enumerate}
\end{definition}

\begin{lemma}
\label{lemma2.7}
Let $G=G(w)$, $H=H(w)$. Then, in $\Lambda[[z]]$:

\begin{enumerate}
\item[(1)]$G=1+H+H^{2}+\cdots$
\item[(2)]$H=Bz+CGAz^{2}$
\end{enumerate}
\end{lemma}

\begin{proof}
Every standard walk from $0$ to $0$ of length $>0$ is either primitive or uniquely a concatenation of two or more primitive standard walks from $0$ to $0$. The multiplicative property of $w$ now gives (1). To prove (2) note that the primitive standard walk $(0,0)$ has $w=B$. And a primitive standard walk from $0$ to $0$ of length $l>1$ is a concatenation of $(0,1)$, a standard walk, $\beta$, from $0$ to $0$ of length $l-2$ and $(0,-1)$. Then $w(\alpha)=C w(\beta)A$. Since $\alpha \rightarrow\beta$ gives a 1--1 correspondence between primitive standard walks of length $l$ from $0$ to $0$ and standard walks of length $l-2$ from $0$ to $0$, we get the result.\qed
\end{proof}

\begin{corollary}
\label{corollary2.8}
If $G=G(w)$, then $G-1-(BG)z-(CGAG)z^{2}=0$ in $\Lambda[[z]]$.

\end{corollary}

\begin{proof}
By (1) of Lemma \ref{lemma2.7}, $(1-H)\cdot G=1$. Substituting $H=Bz+CGAz^{2}$ gives the result.\qed
\end{proof}

\begin{theorem}
\label{theorem2.9}
Suppose that $\Lambda = M_{s}(F)$, $F$ a field, so that $G(w)$ may be viewed as an $s$ by $s$ matrix with entries in $F[[z]]$. Then these matrix entries, $u_{i,j}$, are algebraic over $F(z)$.
\end{theorem}

\begin{proof}
Let $U=|U_{i,j}|$ be an $s$ by $s$ matrix of indeterminates over $F$, and $p_{i,j}$ be the $(i,j)$ entry in $U-I_{s}-(BU)z-(CUAU)z^{2}$. The $p_{i,j}$ are degree 2 polynomials in $U_{1,1},\ldots ,U_{s,s}$ with co-efficients in $F[z]$. By Corollary \ref{corollary2.8}, $p_{i,j}(u_{1,1},\ldots ,u_{s,s})=0$. Now $p_{i,j}=U_{i,j}-\delta_{i,j}-zf_{i,j}(U_{1,1},\ldots ,U_{s,s},z)$ where the $f_{i,j}$ are polynomials with co-efficients in $F$. It follows that the Jacobian matrix of the $p_{i,j}$ with respect to the $U_{i,j}$, evaluated at $(u_{1,1},\ldots ,u_{s,s})$, is congruent to $I_{s^{2}} \mod z$ in the $s^{2}$ by $s^{2}$ matrix ring over $F[[z]]$, and so is invertible. Thus $(u_{1,1},\ldots , u_{s,s})$ is an isolated component of the intersection of the hypersurfaces $p_{i,j}(U_{1,1},\ldots , U_{s,s})=0$, and so its co-ordinates, $u_{1,1},\ldots , u_{s,s}$, are algebraic over $F(z)$. \qed 
\end{proof}

\begin{lemma}
\label{lemma2.10}
$G(w^{*})^{-1}-G(w)^{-1}=(B-D)z$.
\end{lemma}

\begin{proof}
The proof of Lemma \ref{lemma2.7} (1) shows that $G(w^{*})^{-1}=1-H(w^{*})$ with $H(w^{*})$ as in Definition \ref{def2.6}. So it suffices to show that $H(w)-H(w^{*})=(B-D)z$. Now for a primitive walk $\alpha$ of length $>1$ from $0$ to $0$ one cannot have $\alpha_{i-1}=\alpha_{i}=0$, and so $w(\alpha)=w^{*}(\alpha)$. On the other hand, for the primitive walk $(0,0)$, $w=B$ and $w^{*}=D$. This gives the lemma.\qed
\end{proof}

Combining Lemma \ref{lemma2.10} with Theorem \ref{theorem2.9} we get:

\begin{theorem}
\label{theorem2.11}
If $\Lambda=M_{s}(F)$ the matrix entries of the $s$ by $s$ matrix $G(w^{*})$ are algebraic over $F(z)$.
\end{theorem}

Now let $W=|w_{i,j}|$ where $w_{i+1,i}=A$, $w_{i,i+1}=C$, $w_{1,1}=D$, $w_{i,i}=B$ for $i>1$, and all the other $w_{i,j}=0$. In view of Theorem \ref{theorem2.11} the proof of Theorem I will be complete once we show that $G(W)=G(w^{*})$ where $w^{*}$ is the weight function of Definition \ref{def2.3}. The key to this is:

\begin{lemma}
\label{lemma2.12}
For $k\ge 1$ let $u_{k}^{(n)}$ be $\sum w^{*}(\alpha)$, the sum extending over all standard walks of length $n$ from $k-1$ to $0$. Then:

\begin{enumerate}
\item[(1)] $u_{k}^{(0)}=1$ or $0$ according as $k=1$ or $k>1$.
\item[(2)] $u_{1}^{(n+1)}=Du_{1}^{(n)} + Cu_{2}^{(n)}$.
\item[(3)] If $k>1$, $u_{k}^{(n+1)}=Au_{k-1}^{(n)}+Bu_{k}^{(n)}+Cu_{k+1}^{(n)}$.
\end{enumerate}
\end{lemma}

Lemma \ref{lemma2.12} has the following immediate corollaries, with the first proved by induction on $n$.

\begin{corollary}
\label{corollary2.13}
The first column vector in $W^{n}$ is $(u_{1}^{(n)},u_{2}^{(n)},\ldots $.
\end{corollary}

\begin{corollary}
\label{corollary2.14}
The $(1,1)$ co-efficient of $W^{n}$ is $\sum w^{*}(\alpha)$, the sum extending over all standard walks of length $n$ from $0$ to $0$. So $G(W)=G(w^{*})$.
\end{corollary}

It remains to prove Lemma \ref{lemma2.12}. (1) is evident. Let $\alpha$ be a standard walk of length $n$ from $0$ or $1$ to $0$. Then $\beta=(0,\alpha_{0},\ldots , \alpha_{n})$ is a standard walk of length $n+1$ from $0$ to $0$, and $w^{*}(\beta)$ is $Dw^{*}(\alpha)$ in the first case and $Cw^{*}(\alpha)$ in the second. Also each standard walk $\beta$ of length $n+1$ from $0$ to $0$ arises in this way from some $\alpha$; explicitly $\alpha = (\beta_{1}, \ldots , \beta_{n})$. Summing over $\beta$ we get (2). Similarly, suppose that $k>1$ and that $\alpha$ is a standard walk of length $n$ from $k-2$, $k-1$ or $k$ to $0$. Then $\beta = (k-1, \alpha_{0}, \ldots , \alpha_{n})$ is a standard walk of length $n+1$ from $k-1$ to $0$ and $w^{*}(\beta)=Aw^{*}(\alpha)$ in the first case, $Bw^{*}(\alpha)$ in the second, and $Cw^{*}(\alpha)$ in the third. Also, each standard walk $\beta$ of length $n+1$ arises from such an $\alpha$; explicitly $\alpha =(\beta_{1}, \ldots , \beta_{n})$. Summing over $\beta$ we get (3), completing the proof.\qed

\begin{remark}
\label{remark2.15}
To calculate the matrix entries of $G(W)$ explicitly as algebraic functions of $z$ by the method of Theorem \ref{theorem2.9} involves solving a system of $s^{2}$ quadratic equations in $s^{2}$ variables. This isn't practical when $s>2$; in the next section we give a different proof of Theorem \ref{theorem2.9} that is often better adapted to explicit calculations.
\end{remark}

\section{A partial fraction proof of Theorem \ref{theorem2.9}}
\label{section3}

\begin{theorem}
\label{theorem3.1}
$\sum w(\alpha)x^{\alpha_{0}}$, the sum extending over all length $n$ walks (not necessarily standard) with finish $0$, is the element $(Ax+B+Cx^{-1})^{n}$ of $\Lambda[x,x^{-1}]$.
\end{theorem}

\begin{proof}
Denote the sum by $f_{n}$. Since $f_{0}=1$ it's enough to show that $f_{n+1}=(Ax+B+Cx^{-1})f_{n}$. Let $v_{k}^{(n)}$ be the co-efficient of $x^{k}$ in $f_{n}$. Then $v_{k}^{(n)}=\sum w(\alpha)$, the sum extending over all length $n$ walks from $k$ to $0$. The proof of (3) of Lemma \ref{lemma2.12}, using all walks rather than all standard walks, shows that $v_{k}^{(n+1)}=Av_{k-1}^{(n)}+Bv_{k}^{(n)}+Cv_{k+1}^{(n)}$ for all $k$ in $Z$, giving the result. \qed
\end{proof}

\begin{definition}
\label{def3.2}
\hspace{2em}\vspace{-2ex}
\begin{itemize}
\item[] $M_{0}(w)=\sum w(\alpha)z^{l(\alpha)}$, the sum extending over all\, $0$ to $0$ walks.
\item[] $M_{-1}(w)$ is the sum extending over all\, $-1$ to $0$ (or $0$ to $1$) walks.
\item[] $M_{1}$ is the sum extending over all\, $1$ to $0$ (or $0$ to $-1$) walks.
\end{itemize}
\end{definition}

We'll generally omit the $w$ and just write $M_{0}$, $M_{-1}$ or $M_{1}$.

\begin{corollary}
\label{corollary3.3}
Suppose that $i=0$, $-1$ or $1$. Then $M_{i}$ is the co-efficient of $x^{i}$ in the element $\sum_{0}^{\infty}(Ax+B+Cx^{-1})^{n}z^{n}$ of $\Lambda[x,x^{-1}][[z]]$.
\end{corollary}

\begin{definition}
\label{def3.4}
$J_{0}=J_{0}(w)$ is $\sum w(\alpha)z^{l(\alpha)}$, the sum extending over all primitive $0$ to $0$ walks.
\end{definition}

\begin{theorem}
\label{theorem3.5}
\hspace{2em}\vspace{-2ex}
\begin{enumerate}
\item[(1)] $M_{0} = 1 + J_{0}+J_{0}^{2} + \cdots$.
\item[(2)] $G(w) = M_{0} - M_{1}M_{0}^{-1}M_{-1}$.
\end{enumerate}
\end{theorem}

\begin{proof}
(1) follows from the multiplicative property of $w$, as in the proof of Lemma \ref{lemma2.7}. So $M_{0}^{-1}=1-J_{0}$, and (2) asserts that $G(w)=M_{0}+M_{1}J_{0}M_{-1}-M_{1}M_{-1}$. If $\alpha$ is a walk from $0$ to $0$ let $r(\alpha)$ be the number of ways of writing $\alpha$ as a concatenation of a walk from $0$ to $-1$ and a walk from $-1$ to $0$. Also let $r_{1}(\alpha)$ be the number of ways of writing $\alpha$ as a concatenation of a walk from $0$ to $-1$, a primitive walk from $-1$ to $-1$ and a walk from $-1$ to $0$. The multiplicative property of $w$ shows that $M_{0} + M_{1}J_{0}M_{-1}-M_{1}M_{-1}=\sum w(\alpha)(1 + r_{1}(\alpha) - r(\alpha))z^{l(\alpha)}$, the sum extending over all walks from $0$ to $0$. If $\alpha$ is standard, $r_{1}(\alpha)=r(\alpha)=0$. If $\alpha$ is not standard there is an $i$ with $\alpha_{i}=-1$. Let $i_{1}<i_{2}<\cdots < i_{r}$ be those $i$ with $\alpha_{i}=-1$. One sees immediately that $r(\alpha)=r$ and that $r_{1}(\alpha)=r-1$. So $M_{0}+M_{1}J_{0}M_{-1}-M_{1}M_{-1}$ is the sum over the standard walks from $0$ to $0$ of $w(\alpha)z^{l(\alpha)}$, and this is precisely $G(w)$.\qed
\end{proof}

Suppose now that $\Lambda = M_{s}(F)$, $F$ a field, so that $M_{0}$, $M_{1}$ and $M_{-1}$ may be viewed as $s$ by $s$ matrices with entries in $F[[z]]$. Theorem \ref{theorem3.5}, (2), will give a new proof of Theorem \ref{theorem2.9} once we show that these matrix entries are algebraic over $F(z)$. The facts about the matrix entries of $M_{0}$, $M_{1}$ and $M_{-1}$ follow from a standard partial fraction decomposition argument---we'll give our own version.

The algebraic closure of the field of fractions of $F[[z]]$ is a valued field with value group $Q$. let $\Omega$ be the completion of this field and $\ord:\Omega \rightarrow Q\cup \{\infty\}$ be the $\ord$ function in $\Omega$. Let $\Omega^{\prime}$ consist of formal power series $\sum_{-\infty}^{\infty} a_{i}x^{i}$ with $a_{i}\in \Omega$ and $\ord\ a_{i}\rightarrow \infty$ as $|i|\rightarrow \infty$. $\Omega^{\prime}$ has an obvious multiplication and is an overring of $F[x,x^{-1}][[z]]$. $l_{0}$, $l_{1}$ and $l_{-1}$ are the $\Omega$-linear maps $\Omega^{\prime}\rightarrow\Omega$ taking $\sum a_{i}x^{i}$ to $a_{0}$, $a_{1}$ and $a_{-1}$. Note that $\overline{F(z)}$, the algebraic closure of $F(z)$, imbeds in $\Omega$.

\begin{lemma}
\label{lemma3.6}
Suppose $\lambda \in \overline{F(z)}$ with $\ord\ \lambda \ne 0$. Then the element $x-\lambda$ of $\Omega^{\prime}$ is invertible, and for all $k\ge 1$, $(x-\lambda)^{-k}=\sum_{-\infty}^{\infty} a_{i}x^{i}$ in $\Omega^{\prime}$ with the $a_{i}$ in $\overline{F(z)}$. In particular, $l_{0}$, $l_{1}$ and $l_{-1}$ take each $(x-\lambda)^{-k}$ to an element of $\overline{F(z)}$.
\end{lemma}

\begin{proof}
If $\ord\ \lambda >0$, $x-\lambda = x(1-\lambda x^{-1})$ has inverse $x^{-1}(1+\lambda x^{-1}+\lambda^{2}x^{-2}+\cdots )$, while if $\ord\ \lambda <0$, $x-\lambda = -\lambda(1-\lambda^{-1}x)$ has inverse $-\lambda^{-1}(1+\lambda^{-1}x+\lambda^{-2}x^{2}+\cdots )$. \qed
\end{proof}

\begin{lemma}
\label{lemma3.7}
Let $U_{1}$ and $U_{2}$ be elements of $F[z,x]$. Suppose that $U_{2}\equiv x^{s}\mod z$ for some $s$. Then $U_{2}$ has an inverse in $F[x,x^{-1}][[z]]$ and the co-efficients of $x^{0}$, $x^{1}$ and $x^{-1}$ in the element $U_{1}U_{2}^{-1}$ of $F[x,x^{-1}][[z]]$ all lie in $\overline{F(z)}$.

\end{lemma}

\begin{proof}
Write $U_{2}$ as $x^{s}(1-zp)$ with $p$ in $F[x,x^{-1},z]$. Then $x^{-s}(1+zp+z^{2}p^{2}+\cdots )$ is the desired inverse of $U_{2}$. If $\lambda$ in $\Omega$ has $\ord\ 0$ then $1-zp(\lambda,\lambda^{-1},z)$ has $\ord\ 0$ and cannot be $0$. So when we factor $U_{2}$ in $\overline{F(z)}[x]$ as $q\cdot \Pi(x-\lambda_{i})^{c_{i}}$ with $q$ in $F[z]$ and $\lambda_{i}$ in $\overline{F(z)}$, no $\ord\ (\lambda_{i})$ can be $0$. View $U_{1}U_{2}^{-1}$ as an element of $\overline{F(z)}(x)$. As such it is an $\overline{F(z)}$ linear combination of powers of $x$ and powers of the $(x-\lambda_{i})^{-1}$. Since $l_{0}$, $l_{1}$ and $l_{-1}$ are $\Omega$-linear they are $\overline{F(z)}$-linear. Lemma \ref{lemma3.6} then tells us that $U_{1}U_{2}^{-1}$, viewed as an element of $\Omega^{\prime}$, is mapped by each of $l_{0}$, $l_{1}$ and $l_{-1}$ to an element of $\overline{F(z)}$. This completes the proof. \qed
\end{proof}

\begin{lemma}
\label{lemma3.8}
Let $A$, $B$ and $C$ be in $M_{s}(F)$ and $u\in F[x,x^{-1}][[z]]$ be an entry in the matrix $\left(I_{s}-z(Ax+ B+Cx^{-1})\right)^{-1}$. Then the co-efficients of $x^{0}$, $x^{1}$ and $x^{-1}$ in $u$ all lie in $\overline{F(z)}$.
\end{lemma}

\begin{proof}
$u$ may be written as $U_{1}/U_{2}$ where $U_{1}$ and $U_{2}$ are in $F[z,x]$ and $U_{2}=\det\left(xI_{s}-z(Ax^{2}+ Bx+C)\right)$. Then $U_{2}\equiv x^{s}\mod z$, and we apply Lemma \ref{lemma3.7}.\qed
\end{proof}

\begin{corollary}
\label{corollary3.9}
If $\Lambda=M_{s}(F)$, $F$ a field, then the matrix entries of $M_{0}$, $M_{1}$ and $M_{-1}$ are algebraic over $F(z)$. (So by Theorem \ref{theorem3.5} the same is true of the matrix entries of $G(w)$.)
\end{corollary}

\begin{proof}
$\left(I_{s}-z(Ax+ B+Cx^{-1})\right)^{-1}=\sum_{0}^{\infty}(Ax+B+Cx^{-1})^{n}z^{n}$, and we combine Lemma \ref{lemma3.8} with Corollary \ref{corollary3.3}.\qed
\end{proof}

\section{Examples}
\label{section4}

\begin{example}
\label{example4.1}
For $i$, $j$ positive integers define $v_{i,j}$ by:
\vspace{-2ex}
\begin{enumerate}
\item[(1)] $v_{i,j}=1$ if $i-j\in\{-1,0,1\}$.
\item[(2)] $v_{i,j}=1$ if $j=i+3$ and $i$ is odd.
\item[(3)] All other $v_{i,j}$ are $0$.
\end{enumerate}
\end{example}

We calculate $G(V)$ where $V=|v_{i,j}|$. If we take $s=2$, (1) and (2) in the corollary to Theorem I are satisfied, and $D=B=\left(\begin{smallmatrix}
1&1\\
1&1
\end{smallmatrix}\right)$, 
$A=\left(\begin{smallmatrix}
0&1\\
0&0
\end{smallmatrix}\right)$, 
$C=\left(\begin{smallmatrix}
0&1\\
1&0
\end{smallmatrix}\right)$. Let $G=G(w)=G(w^{*})$. $G$ is a $2$ by $2$ matrix  
$\left(\begin{smallmatrix}
g_{1}&g_{2}\\
g_{3}&g_{4}
\end{smallmatrix}\right)$
with entries in $F[[z]]$, and $g_{1}=G(V)$. By Corollary \ref{corollary2.8}, $CGAGz^{2}+BGz-G+I_{2}=0$. Two of the four equations this gives are:
\begin{eqnarray*}
z^{2}g_{1}g_{3}+z(g_{1}+g_{3})-g_{3} &=& 0\\
z^{2}g_{3}^{2}+z(g_{1}+g_{3})-g_{1}+1 &=& 0\\
\end{eqnarray*}
Solving the first equation for $g_{3}$ and substituting in the second we find that $G(V)=g_{1}$ is a root of:

\makebox{$(z^{5}-z^{4})x^{3}+(3z^{4}-4z^{3}+2z^{2})x^{2}+(2z^{3}-4z^{2}+3z-1)x+(z^{2}-2z+1) = 0$.}

\begin{example}
\label{example4.2}
For $i$, $j$ positive integers define $v_{i,j}$ by:
\vspace{-2ex}
\begin{enumerate}
\item[(1)] $v_{i,j}=1$ if $i-j\in\{-1,0,1\}$.
\item[(2)] $v_{i,j}=1$ if $j=i+3$ and $i$ is even.
\item[(3)] All other $v_{i,j}$ are $0$.
\end{enumerate}

\end{example}

We calculate $G(V^{*})$ where $V^{*}=|v_{i,j}|$. Since $v_{2,5}=1$, condition (1) of the corollary to Theorem I is not met when $s=2$, and we instead take $s=4$.

Now $$D=B=
\left(\begin{smallmatrix}
1 &1 &0 &0\\
1 &1 &1 &0\\
0 &1 &1 &1\\
0 &0 &1 &1
\end{smallmatrix}\right)
\qquad
A=
\left(\begin{smallmatrix}
0 &0 &0 &1\\
0 &0 &0 &0\\
0 &0 &0 &0\\
0 &0 &0 &0
\end{smallmatrix}\right)
\qquad \mbox{and}\qquad
C=
\left(\begin{smallmatrix}
0 &0 &0 &0\\
1 &0 &0 &0\\
0 &0 &0 &0\\
1 &0 &1 &0
\end{smallmatrix}\right)
.$$

Let the entries in the first column of the $4$ by $4$ matrix $G=G(w)$ be $a$, $b$, $c$ and $d$. Examining the entries in the first column of the matrix equation $G=BGz+CGAGz^{2}+I_{4}$ we see:
\begin{eqnarray*}
                       a&=&(a+b)z+1\\
                       b&=&(a+b+c)z+bdz^{2}\\
                       c&=&(b+c+d)z\\
                       d&=&(c+d)z+d(a+c)z^{2}
\end{eqnarray*}

Using Maple to eliminate $b$, $c$, and $d$ from this system we find that $a=G(V^{*})$ is a root of:

\begin{eqnarray*}
(z^{2})\cdot(z-1)^{3}\cdot(3z^{2}+3z-2)\cdot & x^{3}&\\
 +(z-1)^{2}\cdot(9z^{4}+6z^{3}-11z^{2}+5z-1)\cdot & x^{2}&\\
 +(2z-1)\cdot(5z^{4}-13z^{2}+9z-2)\cdot & x^{\ }&\\
+(2z-1)^{2}\cdot(z^{2}+2z-1)&=&0.
\end{eqnarray*}

\begin{example}
\label{example4.3}
For $i$, $j$ positive integers define $v_{i,j}$ by:
\vspace{-2ex}
\begin{enumerate}
\item[(1)] $v_{i,j}=1$ if $i-j\in\{-1,1\}$.
\item[(2)] $v_{i,j}=1$ if $i-j\in\{-3,3\}$ and $i\equiv 2 \pod 3$.
\item[(3)] All other $v_{i,j}$ are $0$.
\end{enumerate}
\end{example}

We calculate $G(V)$ where $V=|v_{i,j}|$. Take $s=3$. Then:
$$A=\left(
\begin{smallmatrix}
0 & 0 & 1\\
0 & 1 & 0\\
0 & 0 & 0
\end{smallmatrix}\right)\qquad
B=D=\left(
\begin{smallmatrix}
0 & 1 & 0\\
1 & 0 & 1\\
0 & 1 & 0
\end{smallmatrix}\right)\qquad
C=\left(
\begin{smallmatrix}
0 & 0 & 0\\
0 & 1 & 0\\
1 & 0 & 0
\end{smallmatrix}\right).$$
The determinant of the matrix $xI_{3}-z(Ax^{2}+Bx+C)$ is $-x^{2}(zx^{2}+(3z^{2}-1)x+z)$. The splitting field of this polynomial over $F(z)$ is the extension of $F(z)$ generated by $\sqrt{1-10z^{2}+9z^{4}}$. The arguments of section \ref{section3} show that $M_{0}$, $M_{1}$ and $M_{-1}$ have entries in this extension field. It's not hard to write down these matrices explicitly using the partial-fraction decomposition argument. Theorem \ref{theorem3.5} and a Maple calculation then show that the $(1,1)$ entry in $G(w)$ is $4/(3+z^{2}+\sqrt{1-10z^{2}+9z^{4}})$. Since $D=B$, $G(w^{*})=G(w)$, and this $(1,1)$ entry is the desired $G(V)$.

\section{More algebraic generating functions}
\label{section5}

\begin{definition}
\label{def5.1}
Suppose that $\Lambda = M_{s}(F)$, $F$ a field, and that $A$, $B$, $C$, $D$ are in $\Lambda$.  Then $\mathcal{L}\subset$ the field of fractions of $F[[z]]$ is the extension field of $F(z)$ generated by the matrix entries of the $M_{0}$, $M_{1}$ and $M_{-1}$ of Definition \ref{def3.2}.
\end{definition}

\begin{remark}
\label{remark5.2}
As we've seen $\mathcal{L}$ contains the matrix entries of $G(w)$ and $G(w^{*})$ and is finite over $F(z)$. Indeed the proofs of Lemmas \ref{lemma3.7}, \ref{lemma3.8} and Corollary \ref{corollary3.9} show that $\mathcal{L}\subset$ a splitting field over $F(z)$ of the polynomial $\det |xI_{s}-z(Ax^{2}+Bx+C)|$.  One can say a bit more. The above polynomial splits into linear factors in $\Omega[x]$, and one may view its splitting field as a subfield of the valued field $\Omega$. By examining the partial-fraction decomposition one finds that $\mathcal{L}$ is fixed elementwise by each automorphism of the splitting field that is the identity on $F(z)$ and permutes the roots that have positive $\ord$ among themselves.

\end{remark}
The goal of this section is to show that some generating functions related to $G(w)$ also have their matrix entries in $\mathcal{L}$. These results will be used in a sequel to show the algebraicity (under a conjecture) of certain Hilbert-Kunz series and Hilbert-Kunz multiplicities.

Now let $u_{k}^{(n)}$ be as in Lemma \ref{lemma2.12} where $k$ is a positive integer. By definition, $G^{*}(w)=\sum u_{1}^{(n)}z^{n}$.

\begin{lemma}
\label{lemma5.3}
$\sum_{n}u_{k+1}^{(n)}z^{n}=G(w)(Az)\sum_{n}u_{k}^{(n)}z^{n}$.

\end{lemma}

\begin{proof}
A standard walk from $k$ to $0$ can be written in just one way as the concatenation of a standard walk from $k$ to $k$, the walk $(k,k-1)$ and a standard walk from $k-1$ to $0$.\qed

\end{proof}

\begin{corollary}
\label{corollary5.4}
Fix $k\ge 1$. The generating function arising from the $(k,1)$ entries of the matrices $W^{n}$ has its matrix entries in $\mathcal{L}$.

\end{corollary}

\begin{proof}
Corollary \ref{corollary2.13} shows that this generating function is $\sum_{n}u_{k}^{(n)}z^{n}$, and we use Lemma \ref{lemma5.3} and induction.\qed
\end{proof}

\begin{definition}
\label{def5.5}
$G_{r}^{*}=\sum\binom{\alpha_{0}}{r}w^{*}(\alpha)z^{l(\alpha)}$, the sum extending over all standard walks finishing at 0.

\end{definition}

Evidently $G_{0}^{*}=\sum_{k=0}^{\infty}\sum_{n=0}^{\infty}u_{k+1}^{(n)}z^{n}$. By Lemma \ref{lemma5.3}, this is $$\left(1+G(w)Az+(G(w)Az)^{2}+\cdots\right)G(w^{*}).$$ So:

\begin{lemma}
\label{lemma5.6}
$(1-G(w)Az)G_{0}^{*}=G(w^{*})$.
\end{lemma}

A variant of this is:

\begin{lemma}
\label{lemma5.7}
$(1-G(w)Az)G_{r+1}^{*}=G(w)(Az)G_{r}^{*}$.
\end{lemma}

\begin{proof}
We introduce new weight functions $w|t$ and $w^{*}|t$ as follows. Replace $\Lambda$, $A$ and $C$ by $\Lambda[[t]]$, $A(1+t)$ and $C(1+t)^{-1}$, and let $w|t$ and $w^{*}|t$ be the new $w$ and $w^{*}$ that arise. If $\alpha = (\alpha_{0}, \ldots , \alpha_{l})$ is a walk from $k$ to $0$ then there are $k=\alpha_{0}$ more steps of size $-1$ in the walk than there are steps of size $1$. It follows that $w|t(\alpha)$ and $w^{*}|t(\alpha)$ are $(1+t)^{\alpha_{0}}w(\alpha)$ and $(1+t)^{\alpha_{0}}w^{*}(\alpha)$. In particular, $G(w|t)=G(w)$ and $G(w^{*}|t)=G(w^{*})$. Applying Lemma \ref{lemma5.6} in this new situation we find:
$$\left((1-G(w)Az)-G(w)Azt\right)\left(\sum_{k=0}^{\infty}\sum_{n=0}^{\infty}(1+t)^{k}u_{k+1}^{(n)}z^{n}\right)=G(w^{*}).$$
In particular, the co-efficient of $t^{r+1}$ in the left-hand side of the above equation is $0$. Evaluating this co-efficient we get the lemma. \qed

\end{proof}

\begin{theorem}
\label{theorem5.8}
Let $a_{1},a_{2},\ldots$ be elements of $F$. Suppose there is a polynomial function whose value at $j$ is $a_{j}$ for sufficiently large $j$. Let $R_{n}=\sum_{1}^{\infty}a_{k}u_{k}^{(n)}$. Then all the matrix entries of $\sum R_{n}z^{n}$ lie in $\mathcal{L}$.
\end{theorem}

\begin{proof}
Corollary \ref{corollary5.4} shows that the generating function arising from any single $(j,1)$ entry has matrix entries in $\mathcal{L}$. So we may assume that $j\rightarrow a_{j}$ is a polynomial function. Since any polynomial function is an $F$-linear combination of the functions $j\rightarrow\binom{j-1}{r}$, $r=0,1,2,\ldots$ we may assume $a_{j}=\binom{j-1}{r}$.  But then $\sum R_{n}z^{n}$ is $G_{r}^{*}$, and we use Lemmas \ref{lemma5.6}, \ref{lemma5.7} and induction.\qed
\end{proof}

\begin{corollary}
\label{corollary5.9}
Suppose $V=|v_{i,j}|$, $i,j\ge 1$ is a matrix with entries in $F$ satisfying:

\begin{enumerate}
\item[(1)] $v_{i,j}=0$ whenever $i\le s$ and $j>2s$ or $j\le s$ and $i>2s$.
\item[(2)] $v_{i+s,j+s}=v_{i,j}$ whenever $i+j\ge s+2$.
\item[(3)] The initial $2s$ by $2s$ block in $V$ is 
$\left(\begin{smallmatrix}
D&C\\
A&B
\end{smallmatrix}
\right)$.
\end{enumerate}

Suppose further that $a_{1},a_{2},\ldots$ are in $F$ and that for each $i$, $1\le i \le s$, there is a polynomial function agreeing with $k\rightarrow a_{i+sk}$ for large $k$. Let $v_{i}^{(n)}$ be the $(i,1)$ entry in $V^{n}$. Then $\sum_{i,n} v_{i}^{(n)}a_{i}z^{n}$ is in $\mathcal{L}$.
\end{corollary}

\begin{proof}
Construct $W$ as in the proof of the corollary to Theorem I.  As the first column of $W^{n}$ is $u_{1}^{(n)}, u_{2}^{(n)}, \ldots$ it follows that $v_{i+sk}^{(n)}$ is just the $(i,1)$ entry in the $s$ by $s$ matrix $u_{k+1}^{(n)}$. Theorem \ref{theorem5.8} shows that for each $i$ with $1\le i \le s$, $\sum_{k,n}v_{i+sk}^{(n)}a_{i+sk}z^{n}$ is in $\mathcal{L}$. Summing over $i$ we get the result.\qed
\end{proof}

The following results may seem artificial but they're convenient for our intended applications to Hilbert-Kunz theory.

\begin{lemma}
\label{lemma5.10}
Let $Y$ be a finite dimensional vector space over $F$, $T:Y\rightarrow Y$ and $l:Y\rightarrow F$ linear maps and $y_{1},y_{2},\ldots$ a sequence in $Y$. Let $V$ and $s$ be as in Corollary \ref{corollary5.9}. Suppose that for each $i$, $1\le i \le s$, each co-ordinate of $y_{i+sk}$ with respect to a fixed basis of $Y$ is an eventually polynomial function of $k$. Define $y^{(n)}$ inductively by $y^{(0)}=0$, $y^{(n+1)}=Ty^{(n)}+\sum v_{i}^{(n)}y_{i}$---see Corollary \ref{corollary5.9} for the definition of $v_{i}^{(n)}$. Then $\sum l\left(y^{(n)}\right)z^{n}$ is in $\mathcal{L}$.
\end{lemma}

\begin{proof}
$(I-zT)\sum y^{(n)}z^{n}=\sum_{i,n}v_{i}^{(n)}y_{i}z^{n+1}$. By Corollary \ref{corollary5.9}, all the co-ordinates of $(I-zT)\sum y^{(n)}z^{n}$ with respect to a fixed basis of $Y$ lie in $\mathcal{L}$. Since $\det |I-zT|$ is a non-zero element of $F(z)\subset \mathcal{L}$, the same is true of the co-ordinates of $\sum y^{(n)}z^{n}$, giving the lemma.\qed
\end{proof}

\begin{theorem}
\label{theorem5.11}
Suppose $X$ is a vector space over $F$, $Y$ is a finite dimensional subspace, $T:X\rightarrow X$ is linear with $T(Y)\subset Y$, and $E_{1},E_{2},\ldots$ lie in $X$. Suppose further that $T(E_{j})=\sum v_{i,j}E_{i}+y_{j}$, where $V=|v_{i,j}|$ is as in Lemma \ref{lemma5.10} and $y_{1},y_{2},\ldots$ is a sequence in $Y$ satisfying the condition of Lemma \ref{lemma5.10}. Then if $l:X\rightarrow F$ is linear with each $l(E_{i})=0$, the power series $\sum_{0}^{\infty}l\left(T^{n}(E_{1})\right)z^{n}$ is in $\mathcal{L}$.

\end{theorem}

\begin{proof}
Define $y^{(n)}$ as in Lemma \ref{lemma5.10}. Using the identity $\sum_{j}v_{i,j}v_{j}^{(n)}=v_{i}^{(n+1)}$ and induction we find that $T^{n}(E_{1})=\sum v_{i}^{(n)}E_{i}+y^{(n)}$. So $l\left(T^{n}(E_{1})\right)=l(y^{(n)})$ and we apply Lemma \ref{lemma5.10}.\qed
\end{proof}

The following example is closely related to our calculations in \cite{2}. We'll explain how this and similar examples relate to Hilbert-Kunz theory in a sequel to this paper.

\begin{example}
\label{example5.12}
Suppose $\delta_{1}$ and $\delta_{2}$ are a basis of $Y$, that $y_{1}=6\delta_{1}$ and that $y_{k}=(8k-2)\delta_{1}+\delta_{2}$, $k>1$. Suppose further that $T(\delta_{1})=16\delta_{1}$, $T(\delta_{2})=4\delta_{1}+4\delta_{2}$, $T(E_{1})=E_{1}+E_{2}+y_{1}$, and that $T(E_{k})=E_{k-1}+E_{k+1}+y_{k}$ for $k>1$. Suppose $l:X\rightarrow F$ takes $\delta_{1}$ to $1$, and $\delta_{2}$ and each $E_{k}$ to $0$. We shall calculate the power series $S=\sum l\left(T^{n}(E_{1})\right)z^{n}$ explicitly.
\end{example}

In the above situation, $v_{1,1}=v_{i,i+1}=v_{i+1,i}=1$ and all other $v_{i,j}$ are $0$. So we can take $s=1$, $A=C=D=1$ and $B=0$. Since $s=1$, $v_{k}^{(n)}=u_{k}^{(n)}$. It follows from this and the definition of the $y_{k}$ that $\sum_{k,n}v_{k}^{(n)}y_{k}z^{n+1}=z(8G_{1}^{*}+6G_{0}^{*})\delta_{1}+z(G_{0}^{*}-G(w^{*}))\delta_{2}$.

Now the matrix of $T:Y\rightarrow Y$ on the basis $(\delta_{1},\delta_{2})$ is 
$\left(
\begin{smallmatrix}
16&4\\
0&4
\end{smallmatrix}
\right)$. It follows that the matrix of $I- zT$ is
$\left(
\begin{smallmatrix}
1-16z&-4z\\
0&1-4z
\end{smallmatrix}
\right)$ with inverse
$\frac{1}{(1-16z)(1-4z)}\left(
\begin{smallmatrix}
1-4z&4z\\
0&1-16z
\end{smallmatrix}
\right)$.
Since $S$ is the co-efficient of $\delta_{1}$ in $\sum l(y^{(n)})z^{n}=(I-zT)^{-1}\cdot \sum_{k,n}v_{k}^{(n)}y_{k}z^{n+1}$, the last paragraph shows that $(1-16z)(1-4z)S=(z-4z^{2})(8G_{1}^{*}+6G_{0}^{*})+4z^{2}(G_{0}^{*}-G(w^{*}))$. It only remains to calculate $G(w^{*})$, $G_{0}^{*}$ and $G_{1}^{*}$.

Lemma \ref{lemma2.7} and Corollary \ref{corollary2.8} show that $H(w)=z^{2}G(w)$, and $z^{2}G(w)^{2}-G(w)+1=0$. So $G(w)$ and $H(w)$ are $\frac{1-\sqrt{1-4z^{2}}}{2z^{2}}$ and $\frac{1-\sqrt{1-4z^{2}}}{2}$. Lemma \ref{lemma2.10} then shows $G(w^{*})=\frac{1}{2z(1-2z)}(-1+2z+\sqrt{1-4z^{2}})$. Making use of Lemmas \ref{lemma5.6} and \ref{lemma5.7} we find that $G_{0}^{*}$ and $G_{1}^{*}$ are $\frac{1}{1-2z}$ and $\frac{1}{2(1-2z)^{2}}(-1+2z+\sqrt{1-4z^{2}})$. A brief calculation then gives the explicit formula:
$$(1-16z)(1-4z)(1-2z)^{2}S=4z(1-2z)^{2}+(2z-12z^{2})\sqrt{1-4z^{2}}.$$



\label{}



\end{document}